\documentclass[12pt]{amsart}
\usepackage[dvips]{graphicx}
\usepackage{amsmath,graphics}
\usepackage{amsfonts,amssymb,hyperref}
\usepackage{xypic}
\usepackage{comment}
\specialcomment{proofc}{}{}
\includecomment{proofc}
\theoremstyle{plain}
\newtheorem*{theorem*}{Theorem}
\newtheorem*{lemma*} {Lemma}
\newtheorem*{corollary*} {Corollary}
\newtheorem*{proposition*}{Proposition}
\newtheorem*{conjecture*}{Conjecture}
\newtheorem{theorem}{Theorem}[section]
\newtheorem{lemma}[theorem]{Lemma}
\newtheorem*{theorem1*}{Theorem 1}
\newtheorem*{theorem2*}{Theorem 2}
\newtheorem*{theorem3*}{Theorem 3}

\newtheorem{proposition}[theorem]{Proposition}
\newtheorem{conjecture}[theorem]{Conjecture}

\newcommand\rg{\mathrm{rg}\,}
\newcommand\hg{\mathrm{hg}\,}
\newcommand{\co}{\colon \thinspace}

\theoremstyle{remark}
\newtheorem*{remark}{Remark}

\newtheorem{example*}{Example}
\newtheorem*{claim}{Claim}

\theoremstyle{definition}

\def\op{\operatorname}

\def\G{\Gamma}

 \def\Q{\Bbb{Q}}  \def\Z{\Bbb{Z}}  
\def\N{\Bbb{N}}   \def\ll{\langle} \def\rr{\rangle}
   
\def\vol{\op{Vol}} \def\bp{\begin{pmatrix}}
\def\sm{\setminus} \def\ep{\end{pmatrix}} \def\bn{\begin{enumerate}} 
   \def\en{\end{enumerate}}
\def\ba{\begin{array}} \def\ea{\end{array}}  
     \def\wti{\widetilde}

\def\ker{\op{ker}}\def\be{\begin{equation}} \def\ee{\end{equation}} 
   
 \def\hom{\op{Hom}}  
   
 \def\dim{\op{dim}}

    \def\rk{\op{rk}}

\newcommand{\smfrac}[2]{\mbox{\footnotesize$\displaystyle\frac{#1}{#2}$}} 

\def\ol{\overline}

\def\wti{\widetilde}

\def\what{\widehat}

\def\F{\Bbb{F}}

\def\rg{\op{rg}}
\def\rk{\op{rk}}
\def\hg{\op{hg}}

\begin{document}
\title{Epimorphisms of 3-manifold groups}
\author{Michel Boileau}
\address{
Aix-Marseille Universit\'e, CNRS, Centrale Marseille, I2M, UMR 7373,
13453 Marseille, France}
\email{michel.boileau@cmi.univ-mrs.fr }

\author{Stefan Friedl}
\address{Fakult\"at f\"ur Mathematik\\ Universit\"at Regensburg\\   Germany}
\email{sfriedl@gmail.com}

\begin{abstract}
Let $f\colon M\to N$ be a proper map between two aspherical compact orientable 3-manifolds with empty or toroidal boundary. We assume that $N$ is not a closed graph-manifold. Suppose that $f$  induces an epimorphism on fundamental groups.  We show that $f$ is homotopic to a homeomorphism if one of the following holds:
either for any finite-index subgroup $\Gamma$ of $\pi_1(N)$ the ranks of $\Gamma$ and of $f_*^{-1}(\Gamma)$ agree, or for any finite cover $\wti{N}$ of $N$ the Heegaard genus of $\wti{N}$ and the Heegaard genus of the pull-back cover $\wti{M}$ agree.
\end{abstract}

\keywords{3-manifolds, degree-one maps, Heegaard genus}  \subjclass[2010]{57M10, 57M27}

\maketitle

\section{Introduction}
Let $f\colon M\to N$ be a map between two 3-manifolds.
(Here and throughout the paper all 3-manifolds are understood to be connected, compact and  orientable.)
We say $f$ is \emph{proper} if $f(\partial M)\subset \partial N$. We say $f$ is a \emph{$\pi_1$-epimorphism} if the induced map $f_*\colon \pi_1(M)\to \pi_1(N)$ is an epimorphism. Finally $f$ is called a  \emph{degree-one map} if it is proper and if the induced map on homology $f_*\colon H_3(M,\partial M;\Z)\to H_3(N,\partial N;\Z)$ is an isomorphism. 
  It is well-known, see e.g.\ \cite[Lemma~15.12]{He76}, that a degree-one map is a $\pi_1$-epimorphism.

Given a map $f\colon M\to N$ between two 3-manifolds that is either a $\pi_1$-epimorphism or a degree-one map  it is a long-standing question to find out what extra conditions  ensure that $f$ is in fact homotopic to a homeomorphism.
For example, a celebrated theorem of M.\ Gromov and W.\ Thurston says that if $f\colon M\to N$ is a degree-one map between two  hyperbolic 3-manifolds of the same  volume, then $f$ is homotopic to a homeomorphism, \cite[Chapter~6]{Th79}, see also \cite{BCG95}. This result has been generalized in many directions, see e.g.\ \cite{Som95,Der10,Der12}, by using the  Gromov simplicial volume \cite{Gr82}.

For an aspherical 3-manifold $M$ the volume $\vol(M)$ is defined as the sum of the volumes of the hyperbolic pieces of the geometric decomposition of $M$. In the following we say that a map $f\colon M\to N$ between manifolds is a \emph{$\Z$-homology equivalence}, if $f$ induces isomorphisms of all homology groups. We recall the following theorem of B.\ Perron-P.\ Shalen and P.\ Derbez.

\begin{theorem}\label{thm:ps99}
Let $f\colon M\to N$ be a map between  two 
closed,  aspherical 3-manifolds with the same  volume.
If  for any finite covering $\wti{N}\to N$ $($not necessarily regular$)$ the induced map $\wti{f}\colon\wti{M}\to \wti{N}$ is a $\Z$-homology equivalence, then $f$ is homotopic to a homeomorphism.
\end{theorem}

\begin{remark}\mbox{}
\bn
\item The case of closed graph manifolds was dealt with by B.\ Perron-P.\ Shalen~\cite{PS99}, see also \cite{AF11}. 
\item P.\ Derbez \cite{Der03} stated the theorem for Haken 3-manifolds that are not graph manifolds. The proof of P.\ Derbez applies to all closed 3-manifolds, that are not graph manifolds, for which the Geometrization Conjecture is known. In particular by the work of  G.\ Perelman the theorem applies to all aspherical 3-manifolds that are not graph manifolds.
\item The original theorem of P.\ Derbez is formulated in terms of the simplicial volume, but the simplicial volume for 3-manifolds agrees, up to a fixed constant, with the volume. See e.g.\ \cite{BP92} for details.
\item 
The conclusion of the theorem does not hold without the assumption on the simplicial volume for in \cite{BW96} an example of two 
closed aspherical 3-manifolds $M$ and $N$ is given, such that for any finite covering $\wti{N}\to N$  the induced map $\wti{f}\colon\wti{M}\to \wti{N}$ is a $\Z$-homology equivalence, but such that 
 $\vol(M)>\vol(N)$.
\en
\end{remark}

%

In the following recall that a subgroup $\Gamma\subset \pi$ is called \emph{subnormal} if there exists a chain of subgroups $\pi=\Gamma_0\supset \Gamma_1\supset \dots\supset \Gamma_k=\Gamma$, such that each $\Gamma_i$ is normal in $\Gamma_{i-1}$. 
Given a finitely generated group $\pi$ we denote by $\op{rk}(\pi)$ its rank, i.e.\ the smallest cardinality of a generating set of $\pi$. If $f\colon M\to N$ is a $\pi_1$-epimorphism, then given any finite-index subgroup $\Gamma$ of $\pi_1(N)$ the map $f_*$ induces an epimorphism $f_*^{-1}(\Gamma)\to \Gamma$, in particular the inequality
\[ \rk(f_*^{-1}(\Gamma))\,\geq\, \rk(\Gamma)\]
holds. 
Our first theorem says that equality holds for all finite-index $\Gamma$'s only if $f$ is homotopic to a homeomorphism.

\begin{theorem}\label{mainthm}
Let $f\colon M\to N$ be a proper map between two   aspherical 3-manifolds with empty or toroidal boundary. We assume that $N$  is not a closed graph manifold.
If $f$ is a $\pi_1$-epimorphism and if for every  finite-index subnormal subgroup $\Gamma$ of $\pi_1(N)$ the equality 
\[ \rk\big((f_*)^{-1}(\Gamma)\big)\,=\,\rk(\Gamma)\]
holds, then $f$ is  homotopic to  a homeomorphism.
\end{theorem}

One of the key ingredients of the proof of Theorem~\ref{mainthm} is the virtual fibering theorem, stated as Theorem~\ref{thm:vft}. Only 3-manifolds with empty or toroidal boundary can be (virtually) fibered. This explains the restriction to the case of 3-manifolds with empty or toroidal boundary.

Before we can state our second theorem we need to introduce a few more definitions.
\bn
\item
We say that a covering $p\colon \what{N}\to N$ of a manifold $N$ is \emph{subregular} if the covering $p$ can be written as a composition of coverings $p_i\colon N_i\to N_{i-1}$, $i=1,\dots,k$ with $N_k=\what{N}$ and $N_0=N$, such that each $p_i$ is regular.
\item  Given a  3-manifold $M$ we denote by $h(M)$ the minimal number of one-handles in a handle-composition of $M$ with one zero-handle.
If $M$ is closed, then $h(M)$ equals the \emph{Heegaard genus}, i.e.\ the minimal genus of a Heegaard surface of $M$. If $M$ has non-empty boundary, $h(M)$ is smaller or equal to the minimal genus of a Heegaard splitting of the 3-manifold triad $(M; \emptyset, \partial M)$ as defined by A.~Casson and C.~Mc Gordon~\cite{CG87}, see also M.~Scharlemann~\cite{Scha02}.
\en

Our second theorem gives another  variation on Theorem~\ref{thm:ps99}.

\begin{theorem}\label{mainthm2}
Let $f\colon M\to N$ be a proper map between two   aspherical 3-manifolds  with empty or toroidal boundary. We assume that $N$  is not a closed graph manifold.
If $f$ is a $\pi_1$-epimorphism and if for  every finite subregular cover $\wti{N}$ of $N$ and induced cover $\wti{M}$ the inequality $h(\wti{M})\leq h(\wti{N})$ holds, then $f$ is homotopic to a homeomorphism.
\end{theorem}

Using the examples given in \cite{Li013} and \cite{SW07} of closed aspherical 3-manifolds with arbitrarily large discreapency between the Heegaard genus and the rank of the fundamental group,  it is possible to build examples of $\pi_1$-epimorphism $f\colon M\to N$  between two aspherical closed 3-manifolds where the inequality $h(M) < h(N)$ can be arbitrarily large (see Section \ref{sec:examples}). However to the best of our knowledge  it is still not known whether for every degree-one map $f\colon M\to N$  between two aspherical 3-manifolds the inequality $h(M)\geq h(N)$ holds.

The following result, which is basically a consequence of the proof of Theorem~\ref{mainthm2}, shows that the inequality always holds virtually, even for a proper $\pi_1$-epimorphism.

\begin{proposition}\label{mainprop3}
Let $f\colon M\to N$ be a proper map between two   aspherical 3-manifolds  with empty or toroidal boundary. We assume that $N$  is not a closed graph manifold. If $f$ is a $\pi_1$-epimorphism, then there exists a finite subregular cover $\wti{N}$ of $N$ such that the induced cover $\wti{M}$ satisfies the inequality $h(\wti{M})\geq h(\wti{N})$. See Conjecture~\ref{conj:h-degree-one-map}.
\end{proposition}

\begin{remark}
In Theorems~\ref{mainthm}, \ref{mainthm2} and in Proposition~\ref{mainprop3} we  excluded the case that $N$ is a  closed graph manifold. What we really need for the statements to hold is that $N$ is a 3-manifold that is virtually fibered.
By the Virtual Fibering Theorem of I.~Agol~\cite{Ag08,Ag13} P.~Przytycki--D.~Wise~\cite{PW14, PW12} and D.~Wise~\cite{Wi09,Wi12a,Wi12b} any aspherical 3-manifold that is \emph{not} a closed graph manifold is virtually fibered.
The three results also apply if $N$ is a closed graph manifold that is virtually fibered, e.g.\ if $N$ is a Sol-manifold or if $N$ carries a Riemannian metric of nonpositive sectional curvature \cite{Liu013}. We do not know though whether our results hold for  aspherical graph manifolds that are not virtually fibered. We refer to \cite{LW93,Ne96} for more  examples of graph manifolds that are not virtually fibered.
\end{remark}

This paper is organized as follows. In Section~\ref{section:ranks} we provide the proof of Theorem~\ref{mainthm}. The proof of Theorem~\ref{mainthm2} has many formal similarities with the proof of Theorem~\ref{mainthm2}. In Section~\ref{section:heegaard} we will point out how to modify the proof  of Theorem~\ref{mainthm} to obtain a proof of  Theorem~\ref{mainthm2}.


\subsection*{Acknowledgments}
Work on this paper was supported by the SFB 1085 `Higher Invariants' at the Universit\"at Regensburg funded by the Deutsche Forschungsgemeinschaft (DFG).  We also wish to thank the referee for helpful feedback.

\section{The rank of fundamental groups and $\pi_1$-epimorphisms}\label{section:ranks}

%
%

\subsection{Proof of Theorem~\ref{mainthm} in the fibered case}

Let $M$ be a 3-manifold. Throughout this paper we identify $H^1(M;\Z)$ with $\hom(\pi_1(M),\Z)$. We say a primitive class $\phi\in H^1(M;\Z)=\hom(\pi_1(M),\Z)$ is \emph{fibered} 
 if there exists
 a fibration $p\co M\to S^1$ such that the induced map $p_*\co \pi_1(M)\to \pi_1(S^1)=\Z$ coincides with $\phi$. It follows from Stallings' theorem \cite{St62} (together with the resolution of the Poincar\'e conjecture) that $\phi$ is a fibered class if and only if $\ker(\phi\co \pi_1(M)\to \Z)$ is finitely generated.
 

\begin{proposition}\label{mainthm-fib}
Suppose that $f\colon M\to N$ is a proper $\pi_1$-epimorphism   between two 3-manifolds. We assume
that there exists a primitive class $\phi\in H^1(N;\Z)$ such that $\phi\in H^1(N;\Z)$ and $f^*\phi\in H^1(M;\Z)$ are fibered. If $\rk(\ker(f^*\phi))=\rk(\ker(\phi))$, then  $f_*\colon \pi_1(M)\to \pi_1(N)$ is an isomorphism.
\end{proposition}

\begin{proof}
We denote by $F$ the fiber of $\phi$ and we denote by $E$ the fiber of $f^*\phi$.
These are compact orientable surfaces. 

We consider the following commutative diagram of short exact sequences
\[ \xymatrix@R0.7cm@C1.1cm{ 1\ar[r] &\pi_1(E)\ar@{->>}[d]^-{f_*}\ar[r]&\pi_1(M)\ar@{->>}[d]^-{f_*}\ar[r]^-{\phi\circ f_*}&\Z\ar[r]\ar[d]^-=&0\\
1\ar[r] &\pi_1(F)\ar[r]&\pi_1(N)\ar[r]^-{\phi}&\Z\ar[r]&0.}\]
Since $f_*\colon \pi_1(M)\to \pi_1(N)$ is an epimorphism it follows that the map  on the left is also an epimorphism.
 By our hypothesis we have \[\rk(\pi_1(E))\,=\,\rk(\ker(f^*\phi))\,=\,\rk(\ker(\phi))\,=\,\rk(\pi_1(F)).\] 
We make the following claim.

\begin{claim}
The induced map $\pi_1(E)\to \pi_1(F)$ is an isomorphism.
\end{claim}

First consider the case that $M$ is closed. In that case $E$ is also closed. Together with the observation that 
$f_*\colon \pi_1(E)\to \pi_1(F)$ is an epimorphism and from the observation that $\rk(\pi_1(E))=\rk(\pi_1(F))$ we deduce that $F$ is also closed and that the map $f_*\colon \pi_1(E)\to \pi_1(F)$ is an isomorphism.

Now we consider the case that $M$ has boundary. Then also $E$ has boundary. Since $f$ is a proper map the manifold $N$, and thus also $F$ have boundary. Thus $f_*\colon \pi_1(E)\to \pi_1(F)$ is an epimorphism between free groups of the same rank, thus it is already an isomorphism.
This concludes the proof of the claim.

An elementary argument using the above commutative diagram shows that the induced map $f_*\colon \pi_1(M)\to \pi_1(N)$ is an isomorphism.
\end{proof}

\subsection{Ranks and fibers}

\begin{proposition}\label{prop:rank-fibered}
Let $M$ be a 3-manifold. We write $\pi=\pi_1(M)$. Furthermore let $\phi\in H^1(M;\Z)=\hom(\pi,\Z)$ be a fibered class. We write
\[\pi_n=\ker(\pi_1(M)\xrightarrow{\phi}\Z\to \Z/n).\]
Let $p$ be a prime. For any $n$ we have
\[ \dim_{\F_p}(H_1(\pi_n;\F_p))\leq \rk(\pi_n)\leq 1+\rk(\ker(\phi)).\]
Furthermore, there exists an $n$ such that 
\[ \dim_{\F_p}(H_1(\pi_n;\F_p))=\rk(\pi_n)= 1+\rk(\ker(\phi)).\]
\end{proposition}

\begin{proof}
We denote by $F$ the fiber of the fibration corresponding to $\phi$ and we write $\Gamma=\pi_1(F)$. We can identify $\pi$ with a semidirect product $\ll t\rr\ltimes_{\mu} \G$ in such a way that $\phi(t)=1$ and $\phi|_\Gamma$ is trivial. Here $\mu\colon \G\to \G$ denotes the corresponding automorphism of $\Gamma$, which is just given by the monodromy action on $\Gamma=\pi_1(F)$.

 Given any $n$ we have
\[\pi_n=\ll t^n\rr \ltimes \Gamma.\]
It follows that 
\[ \dim_{\F_p}(H_1(\pi_n;\F_p))\leq \rk(\pi_n)\leq 1+\rk(\Gamma)=1+\rk(\ker(\phi)).\]
Now let $p$ be a prime. The homology group $H_1(\Gamma;\F_p)$ is finite. Thus there exists an $n$ such that $\mu_*^n$ acts like the identity on $H_1(\Gamma;\F_p)$. It follows that 
\[ \ba{rcl}\rk(\pi_n)&=&\rk(\ll t^n\rr \ltimes_{\mu} \Gamma)=\rk(
\ll t\rr \ltimes_{\mu^n} \Gamma)\geq H_1(\ll t\rr \ltimes_{\mu^n} \Gamma;\F_p)\\
&=&\rk(\Z\oplus H_1(\Gamma;\F_p))\\
&=&1+\rk(\Gamma)=1+\rk(\ker(\phi)).\ea\]
Here we used that for the fundamental group of the orientable compact surface $F$ we have $\rk(\pi_1(F))=\dim_{\F_p}(H_1(F;\F_p))$. 
\end{proof}

\subsection{The rank gradient and fiberedness}

Let $\pi$ be a group and $\phi\colon \pi\to \Z$ be a homomorphism. 
We write
\[\pi_n=\ker(\pi \xrightarrow{\phi}\Z\to \Z/n).\]
Following \cite{La05} and \cite{DFV14} we 
 we refer to
\[ \rg(\pi,\phi):=\liminf_{n\to \infty} \smfrac{1}{n}\rk(\pi_n)\]
as the \emph{rank gradient} of $(\pi,\phi)$.

\begin{theorem}\label{mainthm-dfv14}
Let $M$ be a 3-manifold and let $\phi\co \pi_1(M)\to \Z$ be a non-trivial homomorphism.
Then the following two statements are equivalent:
\bn
\item $\phi\in \hom(\pi_1(M),\Z)=H^1(M;\Z)$ is fibered,
\item $\rg(\pi_1(M),\phi)=0$.
\en
\end{theorem}

Here the implication (1) $\Rightarrow$ (2) is an immediate consequence of Proposition~\ref{prop:rank-fibered}. The implication  (2) $\Rightarrow$ (1) is the main result of \cite{DFV14} (see also \cite{De16}). The proof in \cite{DFV14} builds on the fact that 
twisted Alexander polynomials detect fibered manifolds~\cite{FV12}, that proof in turn relies on  the recent results of D.\ Wise \cite{Wi09,Wi12a,Wi12b}.

\subsection{Proof of Theorem~\ref{mainthm}}

Before we provide a proof of Theorem~\ref{mainthm} we recall the Virtual Fibering Theorem of Agol~\cite{Ag08,Ag13}, Przytycki--Wise~\cite{PW14, PW12} and Wise~\cite{Wi09,Wi12a,Wi12b}. (We refer to \cite{AFW15} for the precise references for the Virtual Fibering Theorem.)

\begin{theorem}\label{thm:vft}
Any prime 3-manifold with empty or toroidal boundary that is not a closed graph manifold admits a finite cover that is fibered.
\end{theorem}

Now we are finally in a position to prove Theorem~\ref{mainthm}.

\begin{proof}[Proof of Theorem~\ref{mainthm}]
Suppose that $f\colon M\to N$ is a proper $\pi_1$-epimorphism between two aspherical 3-manifolds  with empty or toroidal boundary.
We assume  $N$  is not a closed graph manifold.
We suppose that the following condition holds:
\bn
\item[$(*)$] Given any  finite-index subnormal subgroup $\Gamma$ of $\pi_1(M)$ the equality 
\[ \rk(f_*^{-1}(\Gamma))=\rk(\Gamma)\]
holds.
\en

We will show that $f$ is homotopic to a homeomorphism.
Since $N$ is a prime 3-manifold  with empty or toroidal boundary that is not a closed graph manifold it follows from Theorem~\ref{thm:vft} that $N$ admits a finite-index regular covering $p\colon \wti{N}\to N$ such that $\wti{N}$ admits a primitive fibered class $\phi\in H^1(\wti{N};\Z)=\hom(\pi_1(\wti{N}),\Z)$.

We denote by $\wti{M}$ the finite-cover of $M$ corresponding to $\pi_1(\wti{M}) = f_{*}^{-1}(\pi_1(\wti{N}))$.
Given $n\in \N$ we denote by $\wti{M}_n$ the cover of $\wti{M}$ corresponding to the epimorphism 
\[ \pi_1(\wti{M})\xrightarrow{\phi\circ f_*}\Z\to \Z/n.\]
Similarly we denote by $\wti{N}_n$ the cover of $\wti{N}$ corresponding to the epimorphism 
\[ \pi_1(\wti{N})\xrightarrow{\phi}\Z\to \Z/n.\]
Evidently $\pi_1(\wti{N}_n)$ is a subnormal subgroup of $\pi_1(N)$ and we have $\pi_1(\wti{M}_n)=f_*^{-1}(\pi_1(\wti{N}_n))$.  By our assumption $(*)$
we have 
\[ \rk\big(\pi_1(\wti{M}_n)\big)=\rk\big(\pi_1(\wti{N}_n)\big) \mbox{ for all $n$.}\]

Since $\phi$ is fibered it follows from Theorem~\ref{mainthm-dfv14} that $\rg(\pi_1(\wti{N}),\phi)=0$. By the above observation we have $\rg(\pi_1(\wti{M}),f^*\phi)=0$. It is a consequence of 
Theorem~\ref{mainthm-dfv14} that $f^*\phi$ is a fibered class of $\wti{M}$.

Similarly, it is a consequence of the above observations and of Proposition~\ref{prop:rank-fibered} that 
\[ \rk(\ker(f^*\phi))=\max\big\{ \hspace{-0.1cm}\rk\hspace{-0.1cm}\big(\pi_1\big(\wti{M})_n\big) - 1\,\big|\,n\in \N\big\}
=\max\big\{ \hspace{-0.1cm}\rk\hspace{-0.1cm}\big(\pi_1(\wti{N}_n)\big) - 1\,\big|\,n\in \N\big\}=\rk(\ker(\phi)).\]
It is now a consequence of Proposition~\ref{mainthm-fib} that   $f_*\colon \pi_1(\wti{M})\to \pi_1(\wti{N})$ is an isomorphism.

Next we show that $f_*\colon \pi_1(M)\to \pi_1(N)$ is already an isomorphism. We consider the following diagram 
\[ \xymatrix@R0.6cm{ 1\ar[r]&\pi_1(\wti{M})\ar[d]_-\cong^-{f_*}\ar[r]&\pi_1(M)\ar[d]^-{f_*}\ar[r]&\pi_1(M)/\pi_1(\wti{M})\ar[d]^-{f_*}_-\cong\ar[r]&1 \\
1\ar[r]&\pi_1(\wti{N})\ar[r]&\pi_1(N)\ar[r]&\pi_1(N)/\pi_1(\wti{N})\ar[r]&1.}\]
This diagram consists of maps between pointed sets (i.e.\ sets with distinguished elements).  The top and bottom sequence are exact in the category of pointed sets. The left and the rightmost maps are isomorphisms of pointed sets. The diagram evidently commutes. Thus it follows from the five-lemma that the middle map is also an isomorphism.

Summarizing we showed that  $f_* \colon \pi_1(M)\to \pi_1(N)$ is an isomorphism. Since $f \colon (M, \partial M) \to (N, \partial N)$ is a proper map, it follows from work of Waldhausen~\cite[Corollary~6.5]{Wa68}\cite[Theorem~13.16]{He76} and from Mostow-Prasad rigidity \cite{Mo68,Pr73} that $f$ is homotopic to a  homeomorphism.
\end{proof}

\section{The Heegaard genus and $\pi_1$-epimorphisms}\label{section:heegaard}

%
%
%
\subsection{Proof of Theorem~\ref{mainthm2}}

In this section we will provide the proof of Theorem~\ref{mainthm2}.
We start out with the following well-known lemma.

\begin{lemma}\label{lem:heegaard-genus}
Let $M$ be a manifold. Then the following hold:
\bn
\item $\rk(\pi_1(M))\leq h(M)$.
\item If $M$ is a fibered 3-manifold with connected fiber $F$, then 
\[ h(M)\,\leq \,b_1(F)+1.\]
\en
\end{lemma}

\begin{proof}
By definition we can endow $M$ with a handle decomposition with one zero-handle and $h(M)$ one-handles. It is thus clear that the free group on $h(M)$ generators surjects onto $\pi_1(M)$.  Therefore $\rk(\pi_1(M))\leq h(M)$. This concludes the proof of (1).

Now we turn to the proof of (2). 
If $M$ is fibered with fiber $F$, then we can write $M=F\times [0,4]/(x,0)\sim(\varphi(x),4)$.
First we consider the case that $M$ is closed.
 Pick two closed disks $P\subset \Sigma$ and $Q\subset \Sigma$ with $P\cap Q=\emptyset$ and $\varphi(Q)\cap P=\emptyset$. It is straightforward to see that 
\[ (\ol{F\sm (P\cup Q)})\times \{0\} \,\cup\, \partial P\times [0,2] \,\cup\,(\ol{F\sm (P\cup \varphi(Q)})\times \{2\}\,\cup\,\partial \varphi(Q)\times [2,4]\] 
splits $M$ into two handlebodies of genus $2g(F)+1=b_1(F)+1$. It follows that we can endow $M$ with a handle-decomposition  with one zero-handle and $2g+1=b_1(F)+1$ one-handles.
Now we consider the case that $M$ has non-empty boundary. We pick a closed disk $P\subset F$. Then the image $X$ of $F\times [0,1]\cup P\times [1,3]\cup F\times [3,4]$ in $M$ is homeomorphic to the handle body $F\times [-1,1]$ together with one extra one-handle.
Put differently, $X$ be written as a handlebody with one zero-handle and $b_1(F)+1$ one-handles. We can view $F$ as obtained from the disk by attaching $b_1(F)$ one-handles. It is now straightforward to see that $M$ can be built out of $X$  by attaching $b_1(F)$ two-handles.
It follows that $h(M)\leq b_1(F)+1$.
\end{proof}

The following proposition plays the r\^ole of Proposition~\ref{prop:rank-fibered}.

\begin{proposition}\label{prop:heegaard-fibered}
Let $M$ be a  3-manifold. Furthermore let $\phi\in H^1(M;\Z)=\hom(\pi,\Z)$ be a primitive fibered class with corresponding fiber $F$. Given $n\in \N$ we denote by $M_n$ the cover of $M$ corresponding to $\ker(\pi_1(M)\xrightarrow{\phi}\Z\to \Z/n)$.
For any $n$ we have
\[ h(M_n) \,\leq\, b_1(F)+1.\]
Furthermore, there exists an $n$ such that 
\[ h(M_n) \,=\, b_1(F)+1.\]
\end{proposition}

\begin{remark}
If $M$ is closed and hyperbolic, then Souto~\cite[Theorem~1.1]{Sou08} showed that there exists an $m$ such that for every $n\geq m$ we have 
$ h(M_n) = 2g(F)+1=b_1(F)+1$. 
\end{remark}

\begin{proof}
We pick a fibration $p\colon M\to S^1$ with fiber $F$ and monodromy $\varphi\colon F\to F$ that corresponds to the given primitive fibered class $\phi$. Since $\phi$ is primitive the fiber $F$ is connected. For any $n$ the fibration $p$ gives rise to a fibration $p_n\colon M_n\to S^1$ with fiber $F$ and monodromy $\varphi^n$. The first statement is now a consequence of
Lemma~\ref{lem:heegaard-genus}.

By Proposition~\ref{prop:rank-fibered} there exists an $n$ such that 
\[ \rk(\pi_1(M_n))\,=\,1+\rk(\pi_1(F)).\]
It follows from Lemma~\ref{lem:heegaard-genus} that 
\[ 1+\rk(\pi_1(F))\,=\,\rk(\pi_1(M_n))\,\leq\, h(M_n)\,\leq\,
1+b_1(F)\,=\,1+\rk(\pi_1(F)).\]
Thus we obtain  that $h(M_n)=b_1(F)+1$.
\end{proof}

Let $M$ be a  3-manifold and let  $\phi\colon \pi_1(M)\to \Z$ be an epimorphism. 
We write $M_n$ for the finite cyclic cover corresponding to 
\[\ker(\pi_1(M)\xrightarrow{\phi}\Z\to \Z/n).\]
Following \cite{La06}  we refer to
\[ \hg(M,\phi):=\liminf_{n\to \infty} \smfrac{1}{n}(2h(M_n)-2))\]
as the \emph{Heegaard gradient} of $(M,\phi)$. 
The following proposition plays the r\^ole of Theorem~\ref{mainthm-dfv14}.

\begin{proposition}\label{prop:heegaard-gradient}
Let $M$ be a 3-manifold and let $\phi\co \pi_1(M)\to \Z$ be an epimorphism.
Then the following two statements are equivalent:
\bn
\item $\phi\in \hom(\pi_1(M),\Z)=H^1(M;\Z)$ is fibered,
\item $\hg(M,\phi)=0$.
\en
\end{proposition}

Here the implication (1) $\Rightarrow$ (2) is an immediate consequence of Proposition~\ref{prop:heegaard-fibered}. 
For the implication (2) $\Rightarrow$ (1) note that 
if $\hg(M,\phi)=0$, then it follows from Lemma~\ref{lem:heegaard-genus} that $\rg(M,\phi)=0$, which in turn implies by Theorem~\ref{mainthm-dfv14} that $\phi$ is fibered. 

\begin{remark} 
The implication (2) $\Rightarrow$ (1) was first proved by M.\ Lackenby~\cite[Theorem~1.11]{La06} for closed hyperbolic 3-manifolds (see also \cite{Ren10,Ren14} for extensions).
The proof of M.\ Lackenby is significantly harder than the proof provided in \cite{DFV14} since the latter proof makes use of the results of D.\ Wise \cite{Wi09,Wi12a,Wi12b} to simplify many arguments.
\end{remark} 

The proof of Theorem~\ref{mainthm2} is now 
almost entirely identical to the proof of Theorem~\ref{mainthm}, we only need to replace the study of ranks of subnormal finite-index subgroups of $\pi_1(M)$ and $\pi_1(N)$ by the  Heegaard genera of the corresponding subregular finite-index covers of $M$ and $N$. In particular we need to replace
 Proposition~\ref{prop:rank-fibered} by 
Proposition~\ref{prop:heegaard-fibered} and we need to replace
 Theorem~\ref{mainthm-dfv14} by Proposition~\ref{prop:heegaard-gradient}.
We leave the straightforward and dull details to the reader.

\subsection{Proof of Proposition~\ref{mainprop3}}
Let $f\colon M\to N$ be a proper $\pi_1$-epimorphism map between two  aspherical 3-manifolds  with empty or toroidal boundary. We assume that $N$  is not a closed graph manifold. 
As in the proof of Theorem~\ref{mainthm} there exists a finite subregular cover $\wti{N}$ of $N$ which is fibered with fiber $F$ and which satisfies 
\[ \rk(\pi_1(\wti{N}))\,=\,1+2\rk(\pi_1(F)).\]
We denote by $\wti{M}$ the induced cover of $M$. The map $f$ gives rise to a proper map $f\colon \wti{M}\to \wti{N}$, which induces an epimorphism of fundamental groups.
It follows from  Lemma~\ref{lem:heegaard-genus} that
\[ h(\wti{M})\,\geq \, \rk(\pi_1(\wti{M}))\,\geq \,\rk(\pi_1(\wti{N}))\,=\,1+\rk(\pi_1(F))\,=\,1+b_1(F)\,\geq\,h(\wti{N}).\]
This concludes the proof of Proposition~\ref{mainprop3}.

\section{Examples}\label{sec:examples}

The goal of this short section is to show the existence of $\pi_1$-epimorphisms $f\colon M\to N$  between two aspherical closed 3-manifolds where the inequality $h(M) < h(N)$ can be arbitrarily large.

\begin{proposition}
Given any $n\in \N$ there exist closed orientable aspherical 3-manifolds $M_n$ and $N_n$ such that there is a $\pi_1$-epimorphism $f\colon M_n\to N_n$ while $h(N_n) > h(M_n) + n$.
\end{proposition}

\begin{proof}
Given $n\in \N$, there exists a closed aspherical 3-manifold $N_n$ such that $h(N_n) - \rk(\pi_1(N_n) \geq n + 3$. The manifold $N_n$ can be hyperbolic or have a non-trivial JSJ with a hyperbolic piece by \cite{Li013}, or it can be a graph manifold by \cite{SW07}.

Given such a manifold $N_n$ there is a $\pi_1$-epimorphism $f \colon X_{n} \to N_n$ where $X_n$ is the connected sum of $\rk(\pi_1(N_n))$ copies of $S^1 \times S^2$.  The manifold $X_n$ can be obtained by taking the double of a handlebody $H_n$ of genus $\rk(\pi_1(N_n))$. Then the Heegaard surface $\partial H_n$ is unique up to isotopy by~\cite{Wa68b}. Let $k \subset X_n$ be a knot lying on $\partial H_n$ and which meets essentially the boundary of each meridian disk of $H_n$. Since $k$ can be put in a 1-bridge position with respect to $\partial H_n$, by~\cite[Theorem 1.6]{Do91} the exterior $E(k) = X_n \setminus \mathcal{N}(k)$ is a compact irreducible 3-manifold. By construction $h(E(k)) = \rk(\pi_1(N_n)) + 1$, see~\cite[Remark (2)]{Do91}. Let $W(k)$ be any Whitehead double of $k$. Then the exterior $E(W(k))$ is obtained from $E(k)$ by gluing along $\partial E(k)$ the exterior $W$ of the Whitehead link. Since $h(W) = 2$, by amalgamating the Heegaard splittings of $E(k)$ and $W$ one gets a Heegaard splitting for $E(W(k))$ of genus  $h(E(k)) + h(W) - 1 = h(E(k)) + 1$, see~\cite{La04}. It follows that $h(E(W(k))) \leq h(E(k)) + 1= \rk(\pi_1(N_n)) + 2$. 
By~\cite{Ha82}  there are only finitely many boundary slopes on $\partial E(W(k))$. This implies in particular that for almost all slopes $\alpha$ on $\partial E(W(k))$ the Dehn fillings of $E(W(k))$ produces a closed aspherical 3-manifold $M$.  Moreover there is a degree-one map $g \colon M \to N_n$, since $W(k)$ is null-homotopic in $X_n$, see~\cite{BW96}. Therefore there is a $\pi_1$-surjective map  $f \circ g \colon M \to N_n$ while $h(M) \leq h(E(W(k))) \leq \rk(\pi_1(N_n)) + 2 \leq h(N_n) -(n + 1)$.
\end{proof}

To the best of our knowledge the following conjecture remains open:

\begin{conjecture}\label{conj:h-degree-one-map}
If there is a degree-one map $f\colon M\to N$  between two aspherical 3-manifolds, does the inequality $h(M)\geq h(N)$ hold?
\end{conjecture}

\end{document}